\documentclass[11pt]{amsart}

\usepackage[french, english]{babel}
\usepackage[T1]{fontenc}
\usepackage[utf8]{inputenc}
\usepackage{lmodern}
\usepackage{amsmath,amssymb,latexsym,amsthm}
\usepackage[shortlabels]{enumitem}
\usepackage{geometry}
\usepackage[abbrev]{amsrefs}
\usepackage{microtype}
\usepackage[unicode]{hyperref}
\usepackage{xcolor}

\newtheorem{theorem}{Theorem}
\newtheorem{lemma}[theorem]{Lemma}
\newtheorem{question}[theorem]{Question}
\newtheorem{corollary}[theorem]{Corollary}
\newtheorem{proposition}[theorem]{Proposition}

\theoremstyle{definition}

\newtheorem*{fpsingconj}{$\Fp$-Singer Conjecture}
\newtheorem*{singconj}{Singer Conjecture}
\newtheorem*{luckconj}{L\"uck Conjecture}
\newtheorem*{bvconj}{Bergeron--Venkatesh Conjecture}
\newtheorem*{acknowledgements}{Acknowledgements}
\newtheorem*{remark}{Remark}

\DeclareMathOperator{\Lk}{Lk}
\DeclareMathOperator{\St}{St}

\newcommand{\ga}{\alpha}
\newcommand{\gs}{\sigma}
\newcommand{\gG}{\Gamma}

\newcommand{\gS}{\Sigma}

\newcommand{\R}{\mathbb{R}}
\newcommand{\Z}{\mathbb{Z}}
\newcommand{\Q}{\mathbb{Q}}
\newcommand{\Fp}{\mathbb{F}_{p}}
\newcommand{\F}{\mathbb{F}}

\title[Torsion and nonpositive curvature]
{Mod p and torsion homology growth in nonpositive curvature}

\author{Grigori Avramidi}
\address{Max Planck Institute for Mathematics\\
Bonn\\
Germany, 53111}
\email{gavramidi@mpim-bonn.mpg.de}

\author{Boris Okun}
\address{Department of Mathematical Sciences\\
University of Wisconsin--Milwaukee\\
Milwaukee,WI 53201}
\email{okun@uwm.edu}

\author{Kevin Schreve}
\address{Department of Mathematics\\University of Chicago\\Chicago, IL 60637}
\email{kschreve@math.uchicago.edu}

\begin{document}\begin{abstract}
	We compute the mod $p$ homology growth of residual sequences of finite index normal subgroups of right-angled Artin groups.
	We find examples where this differs from the rational homology growth, which implies the homology of subgroups in the sequence has lots of torsion.
	More precisely, the homology torsion grows exponentially in the index of the subgroup.
	For odd primes $p$, we construct closed locally CAT(0) manifolds with nonzero mod $p$ homology growth outside the middle dimension.
	These examples show that Singer's conjecture on rational homology growth and L\"uck's conjecture on torsion homology growth are incompatible with each other, so at least one of them must be wrong.
\end{abstract}
\maketitle

This paper is about the growth of homology in regular coverings of finite aspherical complexes $X=B\Gamma$.
We will content ourselves with the situation when the fundamental group $\Gamma=\pi_1X$ is residually finite.
This means there is a nested sequence of finite index normal subgroups $\gG_k\vartriangleleft\Gamma$ with $\cap_{k}\gG_k=1$.
We fix a choice of such a sequence and will be interested in the normalized limits of Betti numbers with coefficients in a field $F$
\[
	b^{(2)}_{i}(\Gamma;F):=\limsup_{k} \frac{b_{i}(B\gG_k; F)}{[\Gamma:\gG_k]}
\]
where $F$ is either $\Q$ or $\Fp$.
When $F=\Q$ then L{\"u}ck's approximation theorem \cite{luck} shows this does not depend on the choice of sequence and can be identified with a more analytically defined $i$-th {\it $L^2$-Betti number} of the universal cover $E\Gamma$.
When $F=\Fp$ we will analogously refer to $b_{i}^{(2)}(\Gamma;\Fp)$ as the \emph{ $\Fp$-$L^2$-Betti number}, even though it does not (as far as we know) have an analytic interpretation and it is not even known whether the $\limsup$ depends on the choice of sequence (we abuse notation by omitting the sequence from $b_{i}^{(2)}(\Gamma;\Fp)$).
Note that if the $\limsup$ is independent of the sequence, then it becomes an honest limit.

For a finite aspherical complex $B\Gamma$ it is easy to see that the mod $p$ $L^2$-Betti number is greater or equal to the ordinary $L^2$-Betti number
\[
	b^{(2)}_i(\Gamma;\Fp)\geq b^{(2)}_i(\Gamma;\Q),
\]
but it might be strictly bigger.
We show that this does---in fact---happen for some right-angled Artin groups.
This seems to have not been observed previously and contradicts a conjecture of L{\"u}ck [Conjecture 3.4, \cite{luck1}] that these numbers are independent of the coefficient field.
When $\Gamma$ is a right-angled Artin group we compute $b_{i}^{(2)}(\Gamma;F)$ completely for any coefficient field, via a residually finite variant of the argument Davis and Leary \cite{davisleary} used to compute the ordinary $L^2$-Betti numbers of such groups.
\begin{theorem}\label{kl2betti}
	Let $A_L$ be a right-angled Artin group with defining flag complex $L$ and $F$ any field (e.g.
	$\Q$ or $\Fp$).
	Then
	\[
		b_{i}^{(2)}(A_L;F)=\bar b_{i-1}(L;F).
	\]
\end{theorem}
Here $\bar b_{i-1}(L;F)$ denotes the reduced Betti number of $L$ with coefficients in $F$.
In particular, the $\limsup$ is actually a limit, it does not depend on the choice of chain but does depend on the characteristic of the coefficient field. 
\begin{corollary}\label{c:tors}
	Suppose that $L$ is a flag triangulation of $\R P^2$.
	\begin{align*}
		b_3^{(2)}(A_{L};\Q)&= 0,\\
		b_3^{(2)}(A_{L};\F_2)&=1.
	\end{align*}
\end{corollary}

The proof of Theorem \ref{kl2betti} also shows that $F$-$L^2$-Betti numbers of finite index subgroups of RAAG's are multiplicative (see Corollary \ref{c:finiteindex}), and we will use this later in Theorem \ref{nopsinger}. In general, it is not known whether the $\Fp$-$L^2$-Betti numbers are multiplicative.

By the universal coefficient theorem, $H_n(X,\F_p)$ is determined by $H_n(X,\Q)$ and $\Z/p$-summands in $H_n(X,\Z)$ and $H_{n-1}(X,\Z)$.
In this case, if $A_L$ is as in Corollary \ref{c:tors}, then since $A_L$ has a $3$-dimensional model for $BA_L$, $H_3(B\gG_k; \Z)$ is torsion-free.
Therefore, this discrepancy between $\Q$ and $\F_2$ homology leads to exponentially growing torsion in homology in degree $2$.
\begin{corollary}
	The group $A_L$ as in Corollary \ref{c:tors} has exponential $H_2$-torsion growth:
	\[
		\limsup_{k} \frac{ \log |H_2(B\gG_k;\Z)_{tors}|}{[A_L:\gG_k]} >0.
	\]
	Furthermore, the rank of the $2$-torsion subgroup of $H_2(B\gG_k;\Z)$ grows linearly in $[A_L:\gG_k]$.
\end{corollary}
While it is conjectured that for arithmetic hyperbolic $3$-manifold groups the torsion in homology grows exponentially in residual chains of congruence covers, this is the first example of a finitely presented group of any sort where one can prove that homology torsion grows exponentially in a residual chain, answering a query of Bergeron for such a group.
By contrast, Abert, Gelander, and Nikolov showed that if $L$ is connected then $H_1$-torsion of $A_L$ grows slower than exponentially \cite{abn}.

For other groups $\Gamma$ the computation of $L^2$-Betti numbers and homology torsion growth is a difficult problem.
A basic vanishing principle which can make computations of $L^2$-Betti numbers simpler is the following conjecture often attributed to Singer.
\begin{singconj}
	Let $M^n$ be a closed aspherical manifold.
	Then
	\[
		b_{i}^{(2)}(\pi_1(M^n); \Q) = 0 \text { for } i \ne \frac{n}{2}.
	\]
\end{singconj}

So in the residually finite setting, the free part of homology should grow sublinearly outside the middle dimension.
A more recent vanishing principle regarding torsion growth, motivated by considerations in number theory, is the following conjecture made by Bergeron and Venkatesh in the context of arithmetic locally symmetric spaces \cite{bv} (see also \cite{berg}).
\begin{bvconj}
	Let $G$ be a semisimple Lie group, $\gG$ a cocompact arithmetic lattice in $G$, and $\gG_k$ a sequence of congruence subgroups with $\cap_{k} \gG_k = 1$.
	Then
	\[
		\limsup_{k}\frac{ \log |H_i(B\gG_k;\Z)_{tors}|}{[\Gamma:\gG_k]}=0
	\]
	unless $i = \frac{\dim(G/K)-1}{2}$. 
\end{bvconj}
\begin{remark}
The conjecture is actually more precise and predicts that the limit is positive in some cases, e.g. when $G$ is $SL(3,\mathbb R)$, $SL(4,\mathbb R)$ or $SO(m,n)$ for $mn$ odd.
\end{remark}

Partially motivated by this conjecture, in \cite{luck2} L{\"u}ck suggested such a vanishing principle could hold quite generally for arbitrary closed aspherical manifolds.
\begin{luckconj}[1.12(2), \cite{luck2}]\label{luckconj}
	Let $M^n$ be a closed aspherical $n$-manifold with residually finite fundamental group.
	Let $\gG_k \vartriangleleft \pi_1(M^n)$ be any normal chain with $\bigcap_{k} \gG_k = 1$.
	If $i\not=(n-1)/2$ then\label{BV}
	\[
		\limsup_{k} \frac{\log|H_{i}(B\gG_k;\Z)_{tors}|} {[\pi_1(M^n):\gG_k]} = 0.
	\]
\end{luckconj}
It is interesting to note that the Singer and L\"uck Conjectures together imply an $\Fp$-version of the Singer conjecture.
\begin{fpsingconj}	
	Let $M^n$ be a closed aspherical $n$-manifold with residually finite fundamental group.
	Then
	\[
		b_{i}^{(2)}(\pi_1(M^n);\Fp) = 0 \text{ for } i \ne \frac{n}{2}.
	\]
\end{fpsingconj}
To see this, suppose we have an $n$-manifold $M^{n}$ with $b_i^{(2)}(\pi_1(M^{n}); \Fp) \ne 0$ for $i \neq n/2$.
By Poincar\'e duality, we can assume $i>n/2$.
The K\"unneth formula implies that $M^n \times M^n \times M^n$ has nontrivial $\Fp$-$L^2$-Betti numbers in dimension $3i$.
Since the Singer Conjecture predicts that $b_{3i}^{(2)}(\pi_1((M^{n})^{3}); \Q) = 0$, the universal coefficient theorem implies exponential homological torsion growth in dimension $3i$ or $3i - 1$, which lies above the middle dimension, contradicting L\"uck's Conjecture.

The $\Fp$-Singer Conjecture is open even for $n = 3$ (but see \cite{ce}, \cite{bv}).
But in high enough dimensions, we show this conjecture is not true for any odd prime $p$.
\begin{theorem}\label{nopsinger}
	For any odd prime $p$, the $\Fp$-Singer Conjecture fails in all odd dimensions $\ge 7$ and all even dimensions $\ge 14$.
\end{theorem}
Our examples are manifolds constructed via right-angled Coxeter groups; in particular they are locally CAT(0), so it follows that the rational homology and torsion homology growth conjectures are incompatible in the CAT(0) setting.
On the other hand, our examples are not locally symmetric so even though the Singer conjecture is known for locally symmetric spaces the Bergeron--Venkatesh conjecture remains open.

Here is a brief outline of our construction.
In \cite{os}, it was shown that if a finite type group $\Gamma$ acts properly on a contractible $n$-manifold and $b_{i}^{(2)}(\Gamma;\Q) \ne 0$ for $i > \frac{n}{2}$, then there is a counterexample to the Singer Conjecture (in some dimension possibly different from $n$.)
We employ a similar strategy here.
Our group is a finite index subgroup of a right-angled Artin group with $b_4^{(2)}(A_L; \Fp) \ne 0$, and the $7$-manifold is going to be the Davis complex corresponding to a right-angled Coxeter group associated to a flag triangulation of a $S^{6}$.

This uses Theorem \ref{kl2betti} and the main result of \cite{ados}.
More precisely, suppose $L=S^2\cup_pD^3$ is a flag triangulation of a complex obtained by gluing a $3$-disk to a $2$-sphere along a degree $p$ map.
Theorem \ref{kl2betti} shows that $b_4^{(2)}(A_L; \Fp) =1$.
Since $H_{3}(L; \mathbb{F}_{2} )=0$, \cite{ados}*{Theorem 5.1} shows that a related flag complex $OL$ (the link of a vertex in the Salvetti complex of $A_{L}$) embeds into a flag triangulation $T$ of $S^6$.
This is where we need $p \ne 2$; interestingly this goes back to the fact that van Kampen's obstruction to embedding $d$-dimensional simplicial complexes into $\R^{2d}$ is an order two invariant.

Now, $A_L$ is commensurable to the right-angled Coxeter group $W_{OL}$ by \cite{dj00}, and $W_{OL}$ is a subgroup of $W_{T}$.
This acts properly on the associated Davis complex, a contractible $7$-manifold, so we obtain the desired proper action for a finite index subgroup of $A_L$.

We then show that the $\Fp$-Singer Conjecture fails for either the right-angled Coxeter group associated to $T$ or to a link of an odd-dimensional simplex in $T$.
In other words, there must be a right-angled Coxeter group counterexample in one of the dimensions 3, 5, or 7.

Taking cartesian products of counterexamples and surface groups, we get counterexamples in all the dimensions stated in the theorem.
In this way, we also get a single closed aspherical manifold contradicting $\Fp$-Singer for a finite collection of primes.
This suggests the following.
\begin{question}
	Given a closed aspherical manifold $M^n$, is there a number $N$ so that for all primes $p > N$, the $\Fp$-Singer Conjecture holds for $M^n$?
\end{question}
Of course, this is at least as difficult as the ordinary Singer conjecture, but it seems interesting (and open) in many cases where the ordinary Singer conjecture is known.
Along the same lines, one can modify the L\"uck conjecture by ignoring the contributions to torsion coming from a finite collection of exceptional primes which should be determined by the geometry of the manifold $M^n$ (akin to how the exceptional primes for right-angled Artin groups are determined by the complex $L$.)
This seems particularly interesting for locally symmetric spaces and might be easier than the Bergeron--Venkatesh conjecture.

\begin{acknowledgements}
	We would like to acknowledge hidden contributions of Mike Davis.
	Many of the ideas in this paper originated in our earlier interactions with Mike.
	We also thank Wolfgang L\"uck for pointing us to his survey article and Shmuel Weinberger for insightful comments on an early draft of the paper.
	The first author would like to thank the Max Planck Institute for Mathematics for its support and excellent working conditions.
	This material is based upon work done while the third author was supported by the National Science Foundation under Award No. 1704364.
\end{acknowledgements}

\section{Right-angled Artin and Coxeter groups}\label{raagracg}

We collect some facts about right-angled Artin groups (RAAG's), right-angled Coxeter groups (RACG's) and relations between one and the other which we will need later.
The philosophy to keep in mind is that RAAGs are the things we can compute, RACGs are the things related to closed aspherical manifolds, and translating from the former to the later involves a bit of (classical) embedding theory.

Let $L$ be a flag complex with vertex set $V$.
The one-skeleton of $L$ determines two group presentations.
A presentation for the RAAG $A_L$ has generators $\{g_v\}_{v\in V}$; there are relations $[g_v, g_{v'}] = 1$ (i.e., $g_v$ and $g_{v'}$ commute) whenever $\{v,v'\}\in L^{(1)}$.
The RACG $W_L$ is the quotient of $A_L$ formed by adding the relations $(g_v)^2=1$, for all $v\in V$.

We now describe a standard classifying space for a RAAG $A_L$.
More precisely, let $T^V$ denote the product $(S^1)^V$.
Each copy of $S^1$ is given a cell structure with one vertex $e_0$ and one edge.
For each simplex $\gs\in L$, $T(\gs)$ denotes the subset of $T^V$ consisting of points $(x_v)_{v\in V}$ such that $x_v=e_0$ whenever $v$ is not a vertex of $\gs$.
So, $T(\gs)$ is a $(\dim \gs +1)$-dimensional standard subtorus of $T^V$.
The \emph{Salvetti complex} for $A_L$ is the subcomplex $BA_L$ of $T^V$ defined as the union of the subtori $T(\gs)$ over all simplices $\gs$ in $L$:
\[
	BA_L:= \bigcup_{\gs\subset L} T(\gs).
\]

The link of the unique vertex is a flag complex of the same dimension as $L$, and is usually denoted $OL$ (and called the \emph{octahedralization} of $L$.)

We now give a similar construction of a classifying space for the commutator subgroup $C_L$ of $W_L$ (which is torsion-free and finite index in $W_L$.)
Let $I^V$ denote the product $([-1,1])^V$.
Each copy of $[-1,1]$ is given a cell structure with two vertices and one edge.
For each simplex $\gs\subset L$, $I(\gs)$ denotes the subset of $I^V$ consisting of those points $(x_v)_{v\in V}$ such that $x_v \in \{\pm 1\}$ whenever $v$ is not a vertex of $\gs$.
So, $I(\gs)$ is a disjoint union of parallel faces of $I^V$ of dimension $\dim \gs +1$.
The \emph{standard classifying space} for $C_L$ is the subcomplex $BC_L$ of $I^V$ defined as the union of the $I(\gs)$ over all simplices $\gs$ in $L$:
\[
	BC_L:= \bigcup_{\gs\in L} I(\gs).
\]

The link of each vertex of $BC_L$ is a copy of $L$.
The universal cover of $BC_L$ is denoted $\Sigma_L$ and called the \emph{Davis complex} of $W_L$.
$(\Z/2)^V$ acts on $I^V$ and preserves the subcomplex $BC_L$.
The lifts of this induced action to $\Sigma_L$ are precisely $W_L$ and we have the exact sequence $$1 \to C_L \to W_L \to (\Z/2)^V \to 1.$$
\begin{lemma}\label{rafacts}
	Let $L$ be a flag complex.
	\begin{enumerate}
		\item $A_L$ is commensurable to $W_{OL}$ \cite{dj00}.
		\item $W_L$ is linear, and hence residually finite.
		Therefore, so is $A_L$.
		\item If $L$ is a triangulation of $S^{n-1}$, then $\Sigma_L$ is a contractible $n$-manifold.
	\end{enumerate}
\end{lemma}
With Davis in \cite{ados}, we studied the minimal dimension of aspherical manifolds with right-angled Artin fundamental groups.
Constructing such manifolds involves embedding right-angled Artin groups $A_L$ into manifold Coxeter groups $W_{S^{n-1}}$.
This boils down to finding PL-embeddings of $OL$ into spheres.
The complexes $OL$ have ``join-like'' properties which make them difficult to embed directly but one can compute when the van Kampen embedding obstruction vanishes for these complexes.
It is a complete obstruction to PL-embedding $d$-complexes in $S^{2d}$, except when $d=2$, and gives the following embedding criterion.
\begin{theorem}[2.2, 5.1 and 5.4,\cite{ados}]\label{vk}
	Suppose $L$ is a $d$-dimensional flag complex, $d \ne 2$.
	Then $OL$ embeds as a full subcomplex into a flag $PL$-triangulation of $S^{2d}$ if and only if $H_d(L; \F_2) = 0$.
\end{theorem}
The prime $2$ plays a special role in this theorem because van Kampen's obstruction looks at what happens to pairs of distinct points under a generic map $OL\to S^{2d}$, which leads to an order two (co)-homological invariant.
Therefore, if $H_d(L;\F_2)=0$ we get an embedding $OL\hookrightarrow S^{2d}$ irrespective of the $\Fp$-homology of $L$ for odd primes $p$.
This observation is key to the proof of Theorem \ref{nopsinger}.

\section{Proof of Theorem \ref{kl2betti}}

 Let $L$ be a flag complex, $\Gamma=A_L$ the right-angled Artin group defined by this complex, $B\Gamma$ its Salvetti complex, and $F$ any field.
By Lemma \ref{rafacts}, we can choose a chain $\gG_k\vartriangleleft \Gamma$ of normal, finite index subgroups with $\bigcap_{k}\gG_k=1$.

Consider the cover of the Salvetti complex $B\Gamma$ by the standard maximal tori $T_\ga$.
Its nerve is a simplex $\Delta$ since all the tori intersect at the base-point.
For a simplex $\sigma$ in $\Delta$, we denote the intersection of the corresponding tori by $T_{\sigma}=\bigcap_{\ga \in \sigma} T_{\ga}$.
We look at the finite cover $B\gG_k$ or equivalently, look at the coefficient module $V=F[\Gamma/\gG_k]$.
The Mayer--Vietoris spectral sequence (see VII.4 \cite{brown}) corresponding to the cover $\{T_\ga\}$ has $E^1$ term
\[
	E^1_{i,j}:=C_i(\Delta;H_j(T_{\sigma};V))\implies H_{i+j}(B\Gamma;V)=H_{i+j}(B\gG_k;F).
\]
The following lemma is crucial.
\begin{lemma}\label{vanishlemma}
	\[
		\lim_{k \to \infty} \frac{\dim_{F} H_j(T_\sigma; F[\Gamma/\gG_k])} {[\Gamma:\gG_k]}=
		\begin{cases}
			1 & \text{if } T_{\sigma}=pt, \text{ and } j=0,\\
			0 & \text{ otherwise.}
		\end{cases}
	\]
\end{lemma}
\begin{proof}
	Since covers of tori are tori, the only way the homology of covers of $T_{\sigma}$ can grow linearly is if the number of components of the preimage of $T_{\sigma}$ in $B\gG_k$ grows linearly in the index.
	Since $\gG_k$ is a residual sequence of normal covers, the number of components grows linearly if and only if $T_{\sigma}$ is a point.
	In more detail, since the cover is normal, the number of components is the ratio of indices $\frac{[\Gamma:\gG_k]}{|\pi_1T_{\sigma}:\pi_1T_{\sigma}\cap\gG_k|}$, and since the sequence $\gG_k$ is residual, the denominator grows with $k$ as long as $\pi_1T_{\sigma}$ is infinite.
\end{proof}
Therefore, up to an error whose dimension is sublinear in the index $[\Gamma:\gG_k]$, the spectral sequence is concentrated on the $E^1_{i,0}$ line.
This implies
\begin{equation}\label{collapse}
	\limsup_{k} \frac{\dim_{F} E^2_{i,0}} {[\Gamma:\gG_k]}=\limsup_{k} \frac{b_i(B\gG_k;F)}{[\Gamma:\gG_k]}.
\end{equation}

Next, we approximate the chain complex $E^1_{i,0}$ by something that we will be able to compute exactly.
For this, set
\[
	V_{\sigma}:=
	\begin{cases}
		F[\Gamma/\gG_k] & \text{if } T_{\sigma}=pt,\\
		0 & \text{otherwise.}
	\end{cases}
\]
Then, the projection $E^1_{i,0}\to C_i(\Delta;V_{\sigma})$ is a chain map and its kernel has dimension that is sublinear in the index $[\Gamma:\gG_k]$.
Therefore
\begin{equation}\label{reduce}
	\limsup_{k} \frac{ \dim_{F} E^2_{i,0}}{[\Gamma:\gG_k]}= \limsup_{k} \frac{ \dim_{F} H_i(\Delta;V_{\sigma})} {[\Gamma:\gG_k]}.
\end{equation}
The quantity in the limit on the right can be computed exactly, just in terms of the topology of $L$.
\begin{lemma}
	\[
		\frac{\dim_{F} H_i(\Delta;V_{\sigma})}{[\Gamma:\gG_k]}=\bar b_{i-1}(L;F).
	\]
\end{lemma}
\begin{proof}[Proof] The complex $\Delta$ has a subcomplex $\mathcal L\subset\Delta$ whose simplices are those intersections of tori that consist of more than one point.
	In other words,
	\[
		\mathcal L:=\{\sigma\subset\Delta\mid V_{\sigma}=0\}.
	\]
	This complex $\mathcal L$ is precisely the nerve of the cover of $L$ by maximal simplices, so it is homotopy equivalent to $L$.
	From the definition of $\mathcal L$ and $V_{\sigma}$ we get the exact sequence of chain complexes
	\[
		0\to C_*(\mathcal L;F)\otimes V\to C_*(\Delta;F)\otimes V\to C_*(\Delta;V_{\sigma})\to 0.
	\]
	Since $\Delta$ is a simplex, this implies
	\[
		H_i(\Delta;V_{\sigma})\cong\bar {H}_{i-1}(L;F)\otimes V.
	\]
	This finishes the proof since $V$ is a $[\Gamma:\gG_k]$-dimensional $F$-vector space.
\end{proof}
The theorem follows from this lemma, together with (\ref{collapse}) and (\ref{reduce}). 
Note that we only used normality of the $\gG_k$ in Lemma \ref{vanishlemma}. 
The proof goes through for chains $\gG_k$ where the number of lifts in $B\gG_k$ of standard $n$-tori in $B\gG$ for $n > 0$ grows sublinearly. 
For example, this occurs for normal chains in a finite index subgroup of $\gG$.
\begin{corollary}\label{c:finiteindex}
For a finite index subgroup $H$ of a RAAG $\gG$ we have $$b_i^{(2)}(H;F) = [\gG:H]b_i^{(2)}(\gG;F).$$
\end{corollary}
\begin{proof}
Let $\gG_k \vartriangleleft H$ be a normal chain.
The cover $BH$ of $B\gG$ has $\le [\gG:H]$ lifts of each torus. For each lifted torus $\tilde T_{\gs}$ in $BH$, the number of lifts in $B\gG_k$ is given by $$\frac{[H:\gG_k]}{|\pi_1\tilde T_{\sigma}:\pi_1\tilde T_{\sigma}\cap\gG_k|},$$ which is sublinear. Hence, $b_i^{(2)}(\gG;F) = \lim \frac{b_{i}(B\gG_k; F)}{[\gG:\gG_k]}$, which implies the multiplicativity formula.
\end{proof}

\section{Mayer--Vietoris sequences for \texorpdfstring{$F$-$L^2$}{F-L2}-Betti numbers of Coxeter groups} 

Let $L$ be a flag complex and $W_L$ the corresponding RACG. Look at a decomposition $L = A \cup_{C} B$ where $A, B$ and hence $C$ are full subcomplexes of $L$.
The Coxeter group $W_L$ splits as an amalgamated product $W_L = W_{A} \ast_{W_{C}} W_{B}$, and our goal in this section is to describe relations between $F$-$L^2$-Betti numbers that arise from such splittings.
Everywhere in this section coefficients are in an arbitrary field $F$ and will be omitted to improve readability.

Let $\gG_k$ be a chain of finite index torsion-free normal subgroups with $\bigcap_{k} \gG_k = 1$ (note that any residual normal chain in $W_L$ is eventually torsion-free.)
Let $\Sigma_L$ be the associated Davis complex for $W_L$, and let $Y_L^{k} = \Sigma_L/\gG_k$.

Given any full subcomplex $A$ of $L$, the RACG $W_A$ is a subgroup of $W_L$, and the corresponding Davis complex $\gS_{A}$ is naturally a subcomplex of $\gS_{L}$.
The stabilizer of $\Sigma_A$ in $W_{L}$ is precisely $W_{A}$.
The $W_L$-orbit of $\Sigma_{A}$ in $\Sigma_L$ is a disjoint union of copies of $\Sigma_{A}$.
The intersections of $\gG_k$ with $W_{A}$ give a corresponding chain of finite index subgroups $W_{A} \cap \gG_k \vartriangleleft W_A$.

We let $Y_A^{k}$ denote the image of this orbit in $Y_{L}^{k}$, so that $Y_A^{k}$ is a disjoint union of $\frac{[W_L:\gG_k]}{[W_{A}:\gG_k \cap W_{A}]}$ copies of $\Sigma_{A}/ W_{A} \cap \gG_k$.
It follows that we can compute $b^{(2)}_{i}(W_{A})$ (with respect to the chain $W_{A} \cap \gG_k$) using $Y_{A}^{k}$:
\[
	b^{(2)}_{i}( W_{A}) =\limsup_{k} \frac{b_{i}(Y_{A}^{k})}{[W_{L}: \gG_k ]}.
\]

Suppose that $L=A\cup_CB$ where $A,B$ and hence $C$ are full subcomplexes of $L$.
We then have a decomposition of spaces:
\[
	Y_L^{k} = Y_{A}^{k} \cup_{Y_{C}^{k}} Y_{B}^{k},
\]
and hence a Mayer--Vietoris sequence
\[
	\dots \to H_{i}(Y_{C}^{k}) \to H_{i}(Y_{A}^{k}) \oplus H_{i}(Y_{B}^{k}) \to H_{i}(Y_{L}^{k}) \to \dots
\]

By the above discussion taking $\limsup$ of dimensions of the homology groups in this sequence divided by $[W_{L}: \gG_k]$ gives $ F$-$L^{2}$-Betti numbers of the corresponding Coxeter groups.
Since $\limsup$ is subadditive, it follows that having $b^{(2)}_{i}=0$ for one of the terms gives the usual inequalities between the nearby terms.

The decomposition we will use is when $A = \St(v)$ is the star of a vertex $v$, $B = L-v$ is its complement and $C = \Lk(v)$ is the link of $v$.
In this case, the Mayer--Vietoris sequence leads to the following inequalities.
\begin{lemma}\label{ineq}
	\leavevmode
	\begin{enumerate}
		\item\label{remove1}
		$b^{(2)}_i(W_{L})\leq b_i^{(2)}(W_{L-v})$ \hspace{0.5cm} if \hspace{0.5cm} $b^{(2)}_{i-1}(W_{\Lk(v)})=0$,		
		\item\label{remove2}
		$b^{(2)}_i(W_{L})\geq b_i^{(2)}(W_{L-v})$ \hspace{0.5cm} if \hspace{0.5cm} $b^{(2)}_{i}(W_{\Lk(v)})=0$.
	\end{enumerate}
\end{lemma}
\begin{proof}
	Removing the vertex $v$ from $L$ gives a Mayer--Vietoris sequence
	\[
		\dots \to H_{i}(Y_{\Lk(v)}^{k}) \xrightarrow{i_{1*}\oplus i_{2*}} H_{i}(Y_{\St(v))}^{k}) \oplus H_{i}(Y_{L - v}^{k}) \to H_{i}(Y_L^{k}) \to H_{i-1}(Y_{\Lk(v)}^{k}) \to \dots
	\]
	The map $i_1:Y^k_{\Lk(v)}\to Y^k_{\St(v)}$ is an inclusion of the form $Y\times\{\pm 1\}\hookrightarrow Y\times[-1,1]$, so $i_{1*}$ maps $H_{i}(Y_{\Lk(v)}^{k}) $ onto $H_{i}(Y_{\St(v)}^{k})$.
	The stated inequalities follow from this.
\end{proof}

Iteratively removing vertices leads to the following lemma.
It lets us reduce dimension by passing from complexes to their links.
\begin{lemma}\label{l:evenodd}
	\leavevmode
	\begin{enumerate}
		\item If $A$ is a flag complex with $b^{(2)}_{i}(W_A) \ne 0$, then there exists a vertex $v \in A$ and a full subcomplex $B$ of $\Lk_{A}(v)$ with $b_{i-1}^{(2)}(W_B) \ne 0$.	
		\item If $L$ is a flag complex with $b^{(2)}_{i}(W_{L}) = 0$ and if $A$ is a full subcomplex of $L$ with $b^{(2)}_{i}(W_A) \ne 0$, then there exists a vertex $v \in L $ and a full subcomplex $B$ of $\Lk_{L}(v)$ with $b_{i}^{(2)}(W_B) \ne 0$.
	\end{enumerate}
\end{lemma}
\begin{proof}
	Assume that all the link terms have $b^{(2)}_{i-1}=0$.
	Then removing vertices from $A$ one at a time until we are left with a single vertex leads, by the first part of Lemma \ref{ineq}, to
	\[
		b^{(2)}_i(W_A)\leq b^{(2)}_i(W_{A-v})\leq\dots\leq b^{(2)}_i(W_{pt})=0
	\]
	which contradicts the assumption that $b_i^{(2)}(W_A)>0$.
	This proves the first part.
	
	Now, assume that all the link terms have $b^{(2)}_i=0$.
	Then removing vertices from $L$ one at a time until we are left with $A$ leads, by the second part of Lemma \ref{ineq}, to
	\[
		b^{(2)}_i(W_L)\geq b^{(2)}_i(W_{L-v})\geq\dots\geq b^{(2)}_i(W_{A})
	\]
	which contradicts the assumption $b^{(2)}_i(W_L)=0<b^{(2)}_i(W_A)$.
	This proves the second part.
\end{proof}

\section{Proof of Theorem \ref{nopsinger}} 

We are now ready for the proof of Theorem \ref{nopsinger}.
Fix an odd prime $p$.
Let $L=S^2\cup_pD^3$ be a flag triangulation of a complex obtained by gluing a $3$-disk to a $2$-sphere via a degree $p$ map.
Since $H_{3} (L; \mathbb{F}_{2}) = 0$, by Theorem \ref{vk} the octahedralization $OL$ embeds as a full subcomplex of a flag $PL$-triangulation $T$ of $S^6$.

Using commensurability we choose a common finite index subgroup $N$ of $A_L$ and $W_{OL}$, which is normal in $W_{OL}$.
We fix a torsion-free normal residual chain in $W_{T}$ which intersects $W_{OL}$ inside $N$.
These are abundant as there is an obvious retraction $r:W_T \rightarrow W_{OL}$ so we can intersect any residual chain with $r^{-1}(N)$. 

Since $\gS_{T}$ is a $7$-manifold, the $\Fp$-Singer Conjecture predicts vanishing of $b^{(2)}_{*}(W_{T};\Fp)$, and similarly, vanishing for the links of odd-dimensional simplices.
\begin{proposition}
	The $\Fp$-Singer Conjecture fails either for $T$, or for one of the links of $1$ or $3$-dimensional simplices.
\end{proposition}
\begin{proof}
	Suppose the $\Fp$-Singer Conjecture holds for $T$, and in particular $b^{(2)}_{4}(W_{T};\Fp)= 0$.
Since the chain in $W_{OL}$ is contained in $N$ and $N$ has finite index in $A_{L}$, Corollary \ref{c:finiteindex} implies $b^{(2)}_{4}(W_{OL};\Fp) \neq 0$.
	By the second part of Lemma \ref{l:evenodd} applied to $OL$ and $T$, there is a full subcomplex $B$ of $\Lk_{T}(v)$ with $b_4^{(2)}(W_B; \Fp) \ne 0$.
	Now we apply the first part of Lemma \ref{l:evenodd} to $B$, to get a full subcomplex $C$ of $\Lk_{B}(u)$ with $b_3^{(2)}(W_{C}; \Fp) \ne 0$.
	Note that $\Lk_{B}(u)$, and therefore $C$, is a full subcomplex of $\Lk_{\Lk_{T}(v)}(u)=\Lk_{T}(uv)\approx S^{4}$.
	
	If the $\Fp$-Singer Conjecture still holds for $\Lk_{T}(uv)$, we can repeat this argument to get get a full subcomplex $D$ of $\Lk_{T}(\gs^{3})\approx S^{2}$ with $b_2^{(2)}(W_{D}; \Fp) \ne 0$.
	Now, the $\Fp$-Singer Conjecture must fail for $\Lk_{T}(\gs^{3})$, because repeating this argument once more produces a subcomplex of $S^{0}$ with $b^{(2)}_{1} \ne 0$, which is clearly impossible.
\end{proof}

It follows that the $\Fp$-Singer Conjecture must fail in at least one of the dimensions $3$, $5$, or $7$.
Since a closed surface $S_g$ with $g \ge 2$ has $b_1^{(2)}(\pi_1(S_g); \Fp) \ne 0$, taking cartesian products between surface groups and our counterexamples gives, via the K\"unneth formula, counterexamples in all odd dimensions $\geq 7$ and all even dimensions $\geq 14$.
\begin{remark}
	
	The reason why we used a $3$-dimensional complex $S^2\cup_p D^3$ instead of a $2$-dimensional complex $S^1\cup_p D^2$ above is twofold.
	First, for $2$-dimensional complexes there are other obstructions to embedding in $S^{4}$ besides the classical van Kampen obstruction~\cite{fkt}.
	Second, in codimension $2$ there is a problem of extending a given triangulation on the complex to a triangulation of $S^{4}$ (the embedding might be locally knotted.)
	If for a flag triangulation $L$ of $S^1\cup_pD^2$ one can exhibit its octahedralization $OL$ as a subcomplex of $S^4$, then our method would yield a $5$-dimensional counterexample to the $\Fp$-Singer Conjecture.
\end{remark}

\begin{remark}
If a manifold $M$ fibers over $S^1$, a cell counting argument shows $b^{(2)}_\ast(\pi_1(M), \Fp) = 0$ for many residual chains of finite index subgroups $\{\gG_k\}$ (the chains whose images in $\pi_1(S^1)$ are also residual). 
In particular, our odd-dimensional counterexamples to the $\Fp$-Singer Conjecture will not even virtually fiber over the circle, as our inductive arguments give nontriviality of $b^{(2)}_\ast(\pi_1(M), \Fp)$ for a sufficiently large number of chains. 
This contradicts a conjecture of Davis and the second author  \cite{do01}*{Conjecture 14.1.5}. 

\end{remark}

\end{document}